%% file: Relative01.tex
\documentclass[10pt]{amsart}

\usepackage{srcltx}  

\usepackage{amsmath,mathtools}
\usepackage{amsfonts,amssymb,amsthm,amscd,latexsym,euscript, mathrsfs}

\usepackage[all]{xy}
\SelectTips{cm}{10}
\usepackage{hyperref} 

\input{symbols.tex}

\title{Homotopy theory of relative simplicial presheaves}
\author{Boris Chorny}

\date{\today} 

\begin{document}
\begin{abstract}
We show that a category \cat M equipped with a model structure defined by a proper, locally small class of orbits \cat O is Quillen equivalent to the category of small relative presheaves $\relpre M{}O$.
\end{abstract}

\maketitle

\section{Introduction}
Characterization of presheaf categories is a classical problem in category theory. Simple criteria determining if a category is equivalent to a category of presheaves were found by M.~Bunge in her Ph.D. thesis. A generalization for presheaves taking values in a closed monoidal category appeared in \cite[Corollary~4.19]{Bunge}. Bunge's theorem states that a category $\cat E$ is equivalent to a presheaf category $\Set^{\cat C^{\op}}$ if and only if it is a complete, wellpowered and co-wellpowered, coregular (epimorphisms are coequalizers) category with a generating set of abstractly unary projectives (atoms). $\cat C$ my be chosen to be the full subcategory of $\cat E$ generated by the atoms.  


A similar question in homotopy theory was treated by W.G.~Dwyer and D.~Kan, \cite{DK}. They  introduced the concept of a subcategory $\cat O$ of \emph{orbits} (instead of atoms) in a simplicial category $\cat M$, such that $\cat M$ equipped with a \emph{set} of orbits carries a model structure Quillen equivalent to the simplicial presheaves $\pre O$ with the projective model structure. Let $\cat S$ denote the category of simplicial sets, which we shall also call spaces, and let $G$ be a (discrete or simplicial) group, then the homogeneous spaces $\cat O_{G}=\{G/H \,|\, H<G\}$ are orbits also in the sense of Dwyer and Kan. In this example the Bredon equivariant model category on $G$-spaces is Quillen equivalent to the category of simplicial presheaves $\pre{O_{\mathit{G}}}=\cat S^{\cat O^{\op}_{G}}$. Conversely, if $\cat D$ is a small simplicial category, then in the category of simplicial presheaves $\pre D$ equipped with the projective model structure, then the representable functors $\{R_{D}=\hom_{\cat D}( - , D) \,|\, D\in \cat D\}$ play the role of the orbits.

The ideas of Dwyer and Kan have many applications, variations and generalizations. 

A set of orbits $\left\{\left.\left(\bigwedge_{i=1}^{n} R^{S^0}\right)_{\text{cof}} \,\right|\, n\in \mathbb{N}\right\}$ in the category of functors from finite pointed spaces to pointed spaces allowed Dwyer and C.~Rezk to classify the polynomial functors in the sense of T.~Goodwillie, \cite{Goo:calc3}, as simplicial presheaves over this orbit category (unpublished). 

In the stable situation, similar considerations lead S.~Schwede and B.~Shipley to classify  spectral model categories with a set of compact generators $\cat G$ as categories of spectral presheaves over the endomorphism category of $\cat G$, i.e., over the category of orbits, \cite{Schwede-Shipley}. 

The construction, by Dwyer and Kan, of a model category determined by a set of orbits, \cite[2.2]{DK}, was generalized by E.~Dror Farjoun, \cite[1.3]{Farjoun}, to encompass model categories determined by proper classes of orbits with an additional property called \emph{local smallness}, see Definition~\ref{local-smallness}, which came up in the generalization of Bredon's equivariant homotopy theory to the arbitrary diagram categories by Farjoun and A.~Zabrodsky, \cite{DZ}. Therefore the resulting model category of $\cat D$-shaped diagrams is called the \emph{equivariant} model structure. 

The purpose of this work is to generalize theorem \cite[3.1]{DK} by Dwyer and Kan establishing a Quillen equivalence of a category $\cat M$ with the model structure defined by a \emph{set} of orbits $\cat O\subset \cat M$ and the category of simplicial presheaves with the projective model structure $\pre O$. We prove this theorem for a category \cat M equipped with a locally small  class of orbits, and use \emph{relative} simplicial presheaves as a Quillen equivalent model category, see Section~\ref{Relative-section}. As an application we show that the model category introduced by Dror Farjoun, \cite{Farjoun}, is Quillen equivalent to the category of \emph{relative} simplicial presheaves, even though the orbit category is no longer small.

One technical condition a model category must satisfy is completeness and cocompleteness. We follow the modern treatments \cite{Hirschhorn} and \cite{Hovey} rather than original Quillen's work, \cite{Quillen}, and demand that the underlying category be complete and cocomplete with respect to all small limits and colimits and not just with respect to the finite ones. On the other hand we do allow the factorizations to be not necessarily functorial, as in Quillen's original definitions and unlike the modern treatments of the subject.

Prior to this result, we have shown with Dwyer,\cite{Chorny-Dwyer}, that the category of maps of spaces with the equivariant model structure is the Quillen equivalent of the category of \emph{small} presheaves indexed by the category spaces (small contravariant functors, which are left Kan extensions from some small subcategory of spaces), since the category of orbits in this case is equivalent to the category of spaces. The technical advantage of the category of maps of spaces, as opposed to the $\cat D$-shaped diagrams of spaces for  arbitrary $\cat D$, is that the category of orbits $\cat O \cong \cat S$ is complete, implying that the category of small presheaves $\pre S$ is also complete by a theorem of B.~Day and S.~Lack, \cite{Lack}. $\pre {\cat O_{\cat D}}$ need not be complete in general, and we use relative presheaves in order to overcome this and other technical difficulties.

The paper is organized as follows. Relative presheaf categories are introduced and discussed in Section~\ref{Relative-section}. In particular, we give sufficient conditions for the existence of the relative model structure and show that it is well-defined, up to equivalence of $\infty$-categories. In Section~\ref{Main-section} we recall the definition of the collection of orbits as it appeared in \cite{Farjoun} and prove our main result, establishing the Quillen equivalence between a model structure defined by a collection $\cat O$ of orbits and the category of $\cat O$-relative presheaves. Section~\ref{Application-section} is devoted to proving that the equivariant model structure on the category of diagrams of spaces is Quillen equivalent to the category of relative presheaves, generalizing \cite{Chorny-Dwyer}. The paper concludes with further examples of relative presheaf categories, which have appeared in recent papers \cite{Duality}, \cite{Blanc-Chorny}, \cite{ClassHomFun}.

\subsection*{Acknowledgments} The idea of using the category of small functors in this context was suggested by E.~Dror Farjoun in the introductory section of \cite{Farjoun}. In this work we implement some of Farjoun's ideas and  are grateful to him for generously sharing them. Marta Bunge and Ross Street made valuable comments concerning the early versions of this manuscript. The hospitality of IMUB (Barcelona) is warmly acknowledged, for this work was finished during the author's visit there.

\section{Relative presheaf categories}\label{Relative-section}
We introduce relative presheaves in this section, offering a solution to the following problem: if \cat M is a complete simplicial category and $\cat O \subset \cat M$ is a full subcategory, which is not necessary complete, find a way to speak about a model structure on \pre O, even though \pre O is not complete. The motivating example is the generalized orbit categories in the sense of Farjoun, \cite{Farjoun}, see also \ref{orbits} for the precise definition. 

As a category, relative presheaves $\relpre M{}O = \pre M = \cat S^{\cat M^{\op}}$ are just the small presheaves over \cat M, but the model structure is ``relative to'' \cat O. Then the category of relative presheaves is complete, since \cat M is complete, \cite{Lack}. If the category of orbits $\cat O$ is small, then the category of relative simplicial presheaves with respect to \cat O is Quillen equivalent to \pre O by the theorem \cite[3.1]{DK} of Dwyer and Kan. 

Still, we can ask to what extent the category of relative presheaves depends on the complete category containing \cat O. In this section we show that it is independent of the choice of embedding, at least up to an equivalence of the corresponding simplicial categories. We were unable to show the existence of a Quillen equivalence.

Since there are other possible applications for the relative presheaves outside of the context offered by the concept of an orbit category, we stop using the suggestive notation and study abstract relative presheaves in this section.

\begin{definition}
Let $\cat C$ be a simplicial category and let $\cat D$ be a full simplicial subcategory of $\cat C$. Denote by $\pre C$ the category of small simplicial presheaves on \cat C, i.e., the category of small contravariant functors $F\colon \cat C^{\op}\to \cat S$. A natural transformation $f\colon F\to G$ between small presheaves $F,G\in \pre C$ is a \cat D-\emph{weak equivalence} (resp., \cat D-\emph{fibration}) if for all $D\in \cat D$ the induced maps $f(D)\colon F(D)\to G(D)$ are weak equivalences (resp., fibrations) of simplicial sets.
\end{definition}

The goal of this section is to work out the conditions on \cat C and \cat D such that there would  exist a model structure on $\pre C$ with \cat D-equivalences as weak equivalences and 
\cat D-fibrations as fibrations.

\begin{definition}
Let $\cat C$ be a simplicial category and let $\cat D$ be a full simplicial subcategory of $\cat C$. Suppose that the \cat D-weak equivalences and \cat D-fibrations form a model category on \pre C with cofibrations defined by the left lifting property, then this model structure is called the \cat D-\emph{relative} model structure and denoted by $\pre{C, D}$.
\end{definition}

\begin{example}
	\begin{enumerate}
	\item If $\cat C=\cat D$, then we obtain the concept of the projective model structure constructed by A.K.~Bousfield and Kan for every small simplicial category, \cite{B-K}, and by the author and Dwyer for the case where \cat C is a (large) complete simplicial category, \cite{Chorny-Dwyer}.
	\item If \cat C is a small simplicial category with a full subcategory \cat D, then $\pre{C, D}$ with the relative model structure is Quillen equivalent to \pre D with the projective model structure by \cite[3.1]{DK}. P.~Balmer and M.~Matthey used this model category (also for non-simplicial categories \cat C) in order to introduce the concept of codescent, \cite{Balmer-Matthey-codescent-I}, which, in turn, was used in the reformulation of Baum-Connes and Farrell-Jones conjectures in model categorical terms, \cite{Balmer-Matthey-reformulation}.
	\item If $\cat C = \Sp^{\op}$ the dual of the category of symmetric spectra and $\cat D\subset \cat C$ the subcategory of fibrant spectra, i.e., cofibrant in $\cat C^{\op}$, then the resulting relative model structure $\pre{C, D}$ is the fibrant-projective model structure constructed in our work with G.~Biedermann, \cite{Duality}.
	\end{enumerate} 
\end{example}

Let us denote by $I_{\cat S}$ and $J_{\cat S}$ respectively the generating cofibrations and the generating trivial cofibrations of simplicial sets \cat S. Given the definition of \cat D-weak equivalences and \cat D-fibrations, we can define the classes of generating cofibrations and generating trivial cofibrations 

\begin{align*}
I_{\cat D} = \{R_{D}\otimes (K\cofib L)\,|\, D\in \cat D, (K\cofib L) \in I_{\cat S} \},\\
J_{\cat D} = \{R_{D}\otimes (K\trivcofib L)\,|\, D\in \cat D, (K\trivcofib L) \in J_{\cat S} \},
\end{align*}
so that \cat D-fibrations are precisely $J_{\cat D}\inj$ and \cat D-trivial fibrations are $I_{\cat D}\inj$ by the standard adjunction argument.

\begin{definition}\label{local-smallness}
Let \cat C be a category and let $\cat D\subset \cat C$ be a full subcategory. \cat D is called a \emph{locally small} subcategory of \cat C if for every  $C\in\cat C$ there exists a \emph{set} of objects $\cat W_{C}$ in \cat D such that for every $D\in \cat D$ there is a factorization of every map $D\to C$ through an object in $\cat W_C$. If $\cat C$ is a simplicial category, then we say that a full locally small subcategory \cat D of \cat C is locally small if the underlying category $\cat D_0$ of \cat D is locally small in the underlying category $\cat C_0$ of \cat C. 
\end{definition}

\begin{example}
\begin{enumerate}
\item 
Every small subcategory is locally small;
\item
Collection of orbits in the category of diagrams of spaces, introduced in \cite{DZ}, is locally small, see Proposition~\ref{orbits-locally-small}; 
\item
In a combinatorial model category \cat M, every set \cat A of cofibrant objects defines the class $C(\cat A)$ of \cat A-colocal objects, which is locally small, since for every $M\in \cat M$ the  cellularization of $M$ may be taken as $\cat W_M=\{CW_{\cat A}M\}$;
\item
If \cat K is a (possibly large) simplicial category, then the full image of \cat K in \pre K under the Yoneda embedding is a locally small subcategory of \pre K by \cite[3.5]{Chorny-Dwyer}.
\end{enumerate}
\end{example}

\begin{remark}
Unfortunately there is no widely accepted terminology for the concept that we call, following \cite{Farjoun}, a locally small subcategory. No confusion with the standard categorical concept of a locally small category is possible, since we do not consider categories with large hom-sets in this paper. In a more categorically oriented work the same concept would be called a subcategory satisfying the \emph{cosolution set condition} or a \emph{cone-coreflective} subcategory, \cite{Chorny-Rosicky-2, Chorny-Rosicky-1}.
\end{remark}

\begin{lemma}\label{transitivity-of-local-smallness}
Let \cat E be a category and let $\cat C \subset \cat E$ be a locally small subcategory. Suppose $\cat D\subset \cat C$ is a locally small subcategory of \cat C. Then $\cat D\subset \cat E$ is locally small.
\end{lemma}

\begin{proof}
Left to the reader.
\end{proof}

\begin{proposition}
Let \cat C be a complete simplicial category and \cat D a full locally small subcategory of \cat C. Then the category of small simplicial presheaves may be equipped with the $\cat D$-relative model structure $\pre {C, D}$.
\end{proposition}
\begin{proof}
The relative model structure on the category of small presheaves is a particular case of the model structure determined by a (possibly large) collection of orbits, \cite[1.3]{Farjoun}.

$\pre{C, D}= \pre C$ as a category is complete by \cite[3.9]{Lack}. The full subcategory $\cat O = \{R_D \, |\, D\in\cat D \}$ of functors representable by the objects of $\cat D$ satisfies all the axioms required from a collection of orbits Q0-Q3, \cite{DK, Farjoun}. We need only verify that \cat O is locally small in \cat C. But the collection of all representable functors $Y\cat C= \{R_C \,|\, C\in\cat C\}$ is locally small in $\pre C$ by \cite[3.5]{Chorny-Dwyer}, and \cat O is  equivalent to \cat D. It is locally small in $Y\cat C \cong \cat C$ as \cat D is locally small in \cat C. By Lemma~\ref{transitivity-of-local-smallness} the category of orbits \cat O is locally small in \pre{C,D}.
\end{proof}

This same category \cat D could be embedded as a locally small subcategory into two different categories, say $\cat C_1$ and $\cat C_2$. For example, if \cat D is a locally small subcategory \cat C, then it is also a locally small subcategory of \pre C. In such a situation it is reasonable to expect that the resulting \cat D-relative model categories $\relpre{C}{1}{D}$ and $\relpre C 2 D$ would be Quillen equivalents. This is the case at least for small \cat D according to Theorem \cite[3.1]{DK}, since both these model categories are Quillen equivalents to $\pre D$ with the projective model structure.

\begin{remark}\label{problems}
Unfortunately for a general large category \cat D it is impossible to use \pre D as a node for the zig-zag Quillen equivalence, like $\relpre C 1 D \to \pre D \leftarrow \relpre C 2 D$, for at least two reasons:
\begin{enumerate}
\item The category \pre D need not be complete (even with respect to finite limits; cf. \cite{Lack}) for incomplete \cat D, so it may not be called a model category;
\item The \emph{orbit point functor} (according to the terminology of \cite{Farjoun}) or the the \emph{singular functor} (according to \cite{DK}) $\cat C_1\to \ \pre D$ need not take values in small presheaves, in other words it is not defined. For example, if \cat D is a complete category, then it is a locally small subcategory of a complete category \pre D. Then the relative model category $\relpre {\pre D}{}D$ exists, but it does not admit an orbit point functor, which would be merely a restriction to \cat D in this case, since a small presheaf on $\pre{D}$ need not be small after it is restricted to \cat D. It is not clear if there is a reasonable adjunction between \pre D and $\relpre C 1 D$ preserving sufficient structure of the model categories.  
\end{enumerate}
\end{remark}

By contrast, the Dwyer-Kan equivalence, cf. \cite{DK2}, of the simplicial localizations of the model categories $\relpre C {1} D$ and $\relpre C {2} D$ with the \cat D-relative model structure may be established directly. In other words, the \cat D-relative model structure is well-defined, depending only on the category \cat D and not on the choice of an embedding into a complete category as a locally small subcategory. 

To think about these categories as elements of a model structure defined by J.~Bergner, \cite{Bergner}, or as $\infty$-categories, \cite{Lurie}, we need to change the universe so that the categories of small presheaves can be considered small. 

Recall that a simplicial functor $F\colon \cat A \to \cat B$ between two simplicial categories is a Dwyer-Kan (DK-)equivalence if
\begin{enumerate}
\item The induced map $\hom_{\cat A}(A_{1}, A_{2})\to \hom_{\cat B}(FA_{1}, FA_{2})$ is a weak equivalence of simplicial sets;
\item The functor induced on the category of components $\pi_{0}F\colon \pi_{0}\cat A \to \pi_{0}\cat B$ is an equivalence of categories. 
\end{enumerate}

\begin{proposition}\label{well-defined}
Let \cat D be a simplicial category. Suppose $i_{1}\colon \cat D \to \cat C_{1}$ and $i_{2}\colon \cat D \to \cat C_{2}$ are fully faithful embeddings of \cat D into complete simplicial categories $\cat C_{1}$ and $\cat C_{2}$. Suppose, additionally, that the images of $i_{1}$ and $i_{2}$ are locally small subcategories of $\cat C_{1}$ and $\cat C_{2}$, respectively. Then the simplicial localizations of the \cat D-relative model structures on the categories $\relpre C 1 D$ and $\relpre C 2 D$ with respect to the \cat D-equivalences are equivalent in the sense of Dwyer and Kan.
\end{proposition}
\begin{proof}
To establish a DK-equivalence of two model categories, it suffices to consider their subcategories of fibrant and cofibrant objects, or even cellular fibrant objects. 

Suppose  $X$ is a cellular space $X= \colim _{i<\lambda} X_{i}$, such that $X_{0}=\emptyset$ and if $X_{i}$ is defined, then $X_{i+1}$ is obtained as a pushout: 
\[
\xymatrix{
K\otimes R_{D}^{\cat C_{1}}
\ar[r]
\ar@{^{(}->}[d] 
							 &  X_{i}
							      \ar[d]\\
L\otimes R_{D}^{\cat C_{1}}
\ar[r]
							&  X_{i+1},\\
}
\]
where $K\cofib L$ is a map from $I_{\cat S}$ and $R_{D}^{\cat C_{1}}=\hom_{\cat C_{1}}(-, D)$ for some $D\in \cat D$. 

For a cellular complex built out of \cat D-cells, the standard way to go from $\relpre C1D$ to $\relpre C2D$ through \pre D (using the singular functor composed with the realization functor defined by Dwyer and Kan, \cite{DK}) preserves the cellular structure. In the current situation we can simply put $FX=Y=\Lan_{i_{2}}i_1^{*} X \in \relpre C2D$, where $\Lan_{i_{2}}$ is the simplicial left Kan extension. We obtain a cellular complex $Y= \colim _{i<\lambda} Y_{i}$, such that $Y_{0}=\emptyset$ and if $Y_{i}$ is defined, then $Y_{i+1}$ is obtained as a pushout: 
\[
\xymatrix{
K\otimes R_{D}^{\cat C_{2}}
\ar[r]
\ar@{^{(}->}[d] 
							 &  Y_{i}
							      \ar[d]\\
L\otimes R_{D}^{\cat C_{2}}
\ar[r]
							&  Y_{i+1},\\
}
\]
where $K\cofib L$ is a map from $I_{\cat S}$ and $R_{D}^{\cat C_{2}}=\hom_{\cat C_{2}}(-, D)$ for $D\in \cat D$. 

To verify that $F$ induces a DK-equivalence of the simplicial localizations of our model categories we have to verify that the maps induced by $F$ on the homotopy function complexes are weak equivalences for all cellular fibrant objects $X,\, X'\in \relpre C1D$:
\begin{multline*}
\hom^{h}_{}(X,X')\simeq \hom_{\relpre C1D}( X, X')\cong \hom_{\relpre C1D}( \Lan_{i_{1}}i_1^*X, \Lan_{i_{1}}i_1^*X')\\ 
\cong \hom_{\pre D}(i_{1}^{*}X_{1}, i_{1}^{*}X_{2})\cong \hom_{\relpre C2D}( \Lan_{i_{2}}i_1^*X, \Lan_{i_{2}}i_1^*X')\\
\simeq \hom^{h}_{\relpre C1D}(FX, FX').
\end{multline*}

In addition, we need to verify that $\pi_{0}F=\pi_{0}\Lan_{i_{2}}i_{1}^{*}$ is an equivalence of the category of components, which are the homotopy categories in this situation. We notice that the functor $\pi_0 G = \pi_0 \Lan_{i_1}i_2^*$ is the inverse of $\pi_0 F$.
\end{proof}

\section{Singular functors and realization functors}\label{Main-section}
In this section we recall the definition of a generalized orbit category, \cite{Farjoun}, and prove our main result generalizing Theorem~\cite[3.1]{DK}.  

\subsection{Collections of orbits}\label{orbits} Let \cat M be a complete and cocomplete simplicial category, tensored and cotensored over simplicial sets (alternatively, we can say that \cat M is complete and cocomplete in the enriched sense, i.e., with respect to weighted limits and colimits). We say that a class $\cat O=\{O_{e}\}_{e\in E}$ of objects of \cat M is a collection of orbits for \cat M if, in addition, the following axioms hold:

\subsection*{Q1} If 
\[
\xymatrix{
K\otimes O_{e}
\ar[r]
\ar@{^{(}->}[d] 
							 &  X_{a}
							      \ar[d]\\
L\otimes O_{e}
\ar[r]
							&  X_{a+1}\\
}
\]
is a pushout diagram in \cat M in which $K \cofib L$ is an element of $I_{\cat S}$ and $e\in E$, then, for every $e'\in E$, the induced diagram in \cat S
\[
\xymatrix{
\hom(O_{e'}, K\otimes O_{e})
\ar[r]
\ar@{^{(}->}[d] 
							 &  \hom(O_{e'}, X_{a})
							      \ar[d]\\
\hom(O_{e'}, L\otimes O_{e})
\ar[r]
							& \hom(O_{e'},  X_{a+1})\\
}
\]
is a homotopy pushout.

\subsection*{Q2} If $X_{0}\to\cdots\to X_{a}\to X_{a}\to \cdots$ is a possibly transfinite sequence of objects and maps in \cat M such that each map $X_{a}\to X_{a+1}$ is induced as in Q1 and such that, for every limit ordinal $b$ involved, one has $X_{b}= \colim_{a<b} X_{a}$, then, for every $e\in E$, the induced map
\[
\colim_{a} \hom(O_{e}, X_{a})\to \hom(O_{e}, \colim_{a}X_{a})\in \cat S
\]
is a weak equivalence.

\subsection*{Q3} The class \cat O is locally small in \cat M and there is a limit ordinal $c$ such that for every sequence $X_{1}\to\cdots\to X_{a}\to X_{a+1}\to\cdots$ as in Q2 which is indexed by the ordinals $<c$ and for every $e\in E$, one has:
\[
\colim_{a} \hom(O_{e}, X_{a})\to \hom(O_{e}, \colim_{a}X_{a})\in \cat S.
\]

These axioms generalize the similar concept of Dwyer and Kan, \cite[2.1]{DK}, and were presented for the first time by Farjoun, \cite{Farjoun}. The first theorem of Dwyer and Kan states that a complete and cocomplete simlicial category \cat M equipped with a \emph{set} of orbits has a model structure, with weak equivalences and fibrations being the maps $f\colon A\to B$ such that $\hom(O_{e},f)\colon \hom(O_{e},A)\to \hom(O_{e},B)$ is a weak equivalence or a fibration, respectively, for all $e\in E$, \cite[Theorem~2.2]{DK}. This theorem was generalized to proper classes of orbits satisfying the modified Q3 axiom (the modification is that \cat O is locally small in \cat M) by Farjoun, \cite[Proposition~1.3]{Farjoun}, establishing the existence of an analogous model structure on \cat M, such that the weak equivalences and fibrations are determined by a possibly proper class of conditions. We call this model structure the \cat O-model structure on \cat M.

The second theorem of Dwyer and Kan, \cite[Theorem~3.1]{DK}, was used numerous times but it had resisted full generalization to the class case so far. It states that \cat M with the model structure defined by \cat O is Quillen-equivalent to the projective model structure on \pre O. In a very specific situation with $\cat M=\cat S^{[2]}$ the category of maps of spaces and $\cat O=\left\{\left.\orbit A \, \right |\, A\in \cat S\right\}\cong \cat S$ the locally small subcategory of orbits was shown to be Quillen equivalent to \pre O in \cite{Chorny-Dwyer}. 

The main result of this work is to generalize \cite[3.1]{DK} resolving the obstacles similar to those enumerated in Remark~\ref{problems} by showing that \cat M with the model structure defined by \cat O is Quillen equivalent to the \cat O-relative model category $\relpre M{}O$.
 
\begin{theorem}\label{main-theorem}
Let \cat M be a complete and cocomplete simplicial category. And let \cat O be a category of orbits for \cat M satisfying Q1-Q3. Then the \cat O-model structure on \cat M is Quillen equivalent to the \cat O-relative model category $\relpre M{}O$.
\end{theorem}
\begin{proof}
Consider the Yoneda embedding functor $Y\colon \cat M\to \pre M$. Since $\cat M$ is cocomplete, $Y$ has a left adjoint, $Z\colon \pre M \to \cat M$, simply a coend construction: $Z(A)= \Id_{\cat M}\otimes_{\cat M}A$, for all $A\in \pre M$. Notice that $Z(A)$ is defined for all small $A$, since it may be rewritten as a left Kan extension of a functor defined on a small subcategory of $\cat M^{\op}$, thus allowing one to compute the above coend over a small category. It is a routine to verify that the two functors are adjoint:
\[
\hom_{\cat M}(\Id_{\cat M}\otimes A, X) \cong \hom(A,X)^{\Id_{\cat M}}\cong \hom(A,X^{\Id_{\cat M}})\cong \hom_{\pre M}(A,Y(X)).
\]

This adjunction is a Quillen pair, since the right adjoint $Y$ preserves both fibrations and trivial fibrations by definition of \cat O-model structure on \cat M and the \cat O- relative model structure on $\pre M = \relpre M{}O$.

It remains to show that the adjunction above is a Quillen equivalence. In other words, for all cofibrant $A\in \relpre M{}O$ and fibrant $X\in \cat M$ a map $A\to YX$ is a weak equivalence if and only if the adjoint map $\Id_{\cat M}\otimes A\to X$ is a weak equivalence in \cat M.

Notice that a map $X\to X'$ is a weak equivalence in \cat M if and only if $YX\to YX'$ is a weak equivalence in $\relpre M{}O$.  But the map $A\to YX$ factors through the unit of the adjunction: $A\to Y(\Id_{\cat M}\otimes A)\to YX$. Therefore it suffices to verify that for every cofibrant $A\in \relpre M{}O$ the unit map $A\to Y(\Id_{\cat M}\otimes A)$ is a weak equivalence.

This verification is performed by cellular induction on the skeleton of $A$. For representable functors it follows from the dual Yoneda lemma, while the attachment of cells requires the axioms Q1 and Q2 to prove that the weak equivalence persists through the whole construction.  
\end{proof}

\section{Application and examples}\label{Application-section}
Dwyer and Kan applied their main theorem, \cite[3.1]{DK}, to construction of a Quillen equivalence between the Bredon equivarint homotopy theory of spaces with a group action and the category of simplicial presheaves over the orbit category (the category of homogeneous spaces in this case), \cite[1.2]{DK}.

Farjoun applied the generalization of the concept of orbits in a model category to construct the equivariant model structure on the category of diagrams of spaces, \cite[2.2]{Farjoun}, corresponding to the generalization of Bredon's equivariant homotopy theory to the arbitrary diagram case, previously developed by Farjoun and Zabrodsky, \cite{DZ}.

\subsection{Application} In this section we construct a Quillen equivalence between the equivariant model structure on the category of $\cat D$-shaped diagrams of simplicial sets for a small simplicial category $\cat D$ and the relative simplicial presheaves $\relpre {S^{\cat D}}{}{O_{\cat D}}$, where the orbit category $\cat O_{\cat D}\subset \cat S^{\cat D}$ is the full simplicial subcategory of all orbits in the sense of Farjoun and Zabrodsky, \cite{DZ}, i.e., the diagrams $\dgrm T\in \cat S^{\cat D}$ such that $\colim_{\cat D} \dgrm T=\ast$.

Existence of this Quillen equivalence is an immediate conclusion of Theorem~\ref{main-theorem}. Therefore we can summarize the references to the verification of the axioms Q1-Q3 for $\cat O_D$. While Q2 was proven in \cite[2]{Farjoun}, the verifications of Q1 and Q3 were left to the reader. The proof of Q1 appears in \cite[1.3]{PhDI}. Axiom Q3 consists of the local smallness condition for the category of orbits and the smallness condition for every orbit for a fixed cardinal $c$. All orbits are $\aleph_0$-small with respect to the cellular maps defined in Q2; see \cite[3.1]{PhDI} for the proof. The proof of local smallness of orbits is hidden in \cite[1.3]{Farjoun} together with a generalization of a small object argument. We would like to make it explicit for the sake of completeness of the current paper.

\begin{proposition}\label{orbits-locally-small}
For any diagram $\dgrm X \in \cat S^{\cat D}$ there exists a set of orbits $W_{\dgrm X}\subset \cat O_{\cat D}$, such that for every orbit $\dgrm T$ a map $\dgrm T \to \dgrm X$ factorizes through an element of $W_{\dgrm X}$. In other words, the collection of $\cat D$-orbits forms a locally small class.  
\end{proposition}
\begin{proof}
For every map $x\colon * \to \colim_{\cat D} \dgrm X$ let $\dgrm T_x$ be an orbit over $x$, i.e., $\dgrm T_x = \ast \times_{\colim_{\cat D} \dgrm X} \dgrm X$, where $\ast$ and $\colim_{\cat D} \dgrm X$ are considered constant diagrams over $\cat D$. Then we define $W_{\dgrm X}=\{\dgrm T_x \,|\, x\colon * \to \colim_{\cat D} \dgrm X\}$ as the collection of all orbits in \dgrm X.

The universal property of the pullback implies the required factorization property.
\end{proof}

\subsection{Examples}
The concept of the relative model structure was previously studied by Balmer and Matthey, \cite{Balmer-Matthey-codescent-I,Balmer-Matthey-codescent-II}, and served for the definition of the notion of codescent: a functor $X\in \relpre C{}D$ satisfies codescent if the cofibrant replacement map in the $\cat D$-relative model structure is a levelwise weak equivalence, i.e., for all $C\in \cat C$ the map $\tilde{X}(C)\to X(C)$ is a weak equivalence. Although the definition of codescent applies also to large categories, the prime interest of Balmer and Matthey lies in the study of the examples related to Baum-Connes and Farrell-Jones conjectures, e.g., $\cat C=\cat O_G$ for a group $G$ and $\cat D$ the subcategory of orbits with finite stabilizers. The existence of the relative model structure in this situation follows from \cite[3.1]{DK}. This theory was developed with the purpose of giving a model-theoretical reformulation of the isomorphism conjectures, \cite{Balmer-Matthey-reformulation}.

Recently several new model categories with proper classes of orbits have appeared in the study of the extension of Spanier-Whitehead duality, \cite{Duality}, and in the attempt to develop a recognition principle for the mapping algebras in model categorical terms, \cite{Blanc-Chorny}. Here is a generalized treatment of these examples.

Let \cat M be a simplicial model category, and let $\cat T\subset \cat M$ be a locally small simplicial subcategory. Then the collection $\cat O = \{R_T \,|\, T\in \cat T\}$ of representable functors forms a category of orbits for $\pre M$. The local smallness of $\cat O$ follows from Lemma~\ref{transitivity-of-local-smallness} and \cite[Lemma~3.5]{Chorny-Dwyer}. The remaining conditions Q1-Q3 follow immediately from Yoneda's lemma. Therefore $\relpre {\pre M}{}{\cat O}$ carries the \cat O-relative model structure, which is Quillen equivalent to a relative model structure on an even bigger category $\pre{\pre M}$ by Theorem~\ref{main-theorem}.

There are many examples of locally small classes of objects in a model category. For instance the subcategory of cofibrant objects or, more generally, \cat A-cellular objects for a set $\cat A\subset \cat M$, \cite{Blanc-Chorny}. If we take $\cat M = \Sp^{\op}$ and \cat T the category of fibrant spectra, then we obtain the fibrant-projective model structure on the category of small functors from spectra to spectra used in \cite{Duality} and \cite{ClassHomFun}.

\bibliographystyle{abbrv}
\bibliography{Xbib}

\end{document}

%% file: symbols.tex
\newcommand{\inj}{\text{-inj}}

\newcommand{\Sp}{\ensuremath{\textup{Sp}}}

\newcommand{\Set}{\ensuremath{\textup{\bf Set}}}

\newcommand{\op}{{\ensuremath{\textup{op}}}}

\newcommand{\dgrm}[1]{\ensuremath{\smash{\underset{\widetilde{\hphantom{#1}}}{#1}} \mathstrut}}

\newcommand {\cofib} {\ensuremath{\hookrightarrow}}

\newcommand {\trivcofib} {\ensuremath{\xhookrightarrow \sim}}

\newcommand{\orbit}[1]{\ensuremath{\overset {#1} {\underset \ast \downarrow}}}

\newtheorem {theorem1}{Theorem}[section]
\newtheorem {theorem}[theorem1]{Theorem}

\newtheorem {proposition}[theorem1]{Proposition}
\newtheorem {lemma}[theorem1]{Lemma}
\theoremstyle{definition}
\newtheorem {definition}[theorem1]{Definition}

\newtheorem {example}[theorem1]{Example}
\theoremstyle{remark}
\newtheorem {remark}[theorem1]{Remark}

\newcommand{\cat}[1]{\ensuremath{\mathscr #1}}  
\newcommand{\pre}[1]{\ensuremath{\mathcal{P}(\mathscr{#1}})}  
\newcommand{\relpre}[3]{\ensuremath{\mathcal{P}(\mathscr{#1}_{#2}},\mathscr{#3})}  

\newcommand{\colim}{\ensuremath{\mathop{\textup{colim}}}}

\def\Lan{\ensuremath{\textup{Lan}}}

\newcommand{\Id}{\ensuremath{\textup{Id}}}

\renewcommand{\hom}{\ensuremath{{\rm hom}}}

\newcounter{zahl}%
    {\end{list}}%